\documentclass{article}
\usepackage{amsmath,amssymb,amsthm}
\newtheorem{theorem}{Theorem}
\newtheorem{lemma}{Lemma}
\newtheorem{corollary}{Corollary}
\theoremstyle{definition}
\newtheorem{definition}{Definition}
\theoremstyle{remark}

\newtheorem*{ack}{Acknowledgement}

\DeclareMathOperator{\ev}{ev}
\DeclareMathOperator{\Id}{Id}
\DeclareMathOperator{\Hom}{Hom}
\DeclareMathOperator{\End}{End}
\DeclareMathOperator{\real}{Re}
\DeclareMathOperator{\Img}{Im}
\DeclareMathOperator{\tr}{tr}
\allowdisplaybreaks
\begin{document}
\title{Darboux transforms of a harmonic inverse mean curvature surface}
\author{Katsuhiro Moriya\footnote{Division of Mathematics, Faculty of Pure and Applied Sciences, 
University of Tsukuba, 
 1-1-1 Tennodai, Tsukuba, Ibaraki 305-8571, Japan}
}
\date{}
\maketitle
\abstract{
The notion of a generalized harmonic inverse mean curvature surface in the Euclidean four-space 
is introduced. 
A backward B\"{a}cklund transform 
of a generalized harmonic inverse mean curvature 
surface is defined. 
A Darboux transform of a generalized harmonic inverse mean curvature surface is constructed by a backward B\"{a}cklund transform. 
For a given isothermic harmonic inverse mean curvature surface, its classical Darboux transform is a harmonic inverse mean curvature surface. 
Then a transform of a solution to the Painlev\'{e} III equation in trigonometric form is defined by 
a classical Darboux transform of a harmonic inverse mean curvature surface of revolution. 
}
\section{Introduction}
The theory of surfaces is connected with the theory of solitons through 
a compatibility condition of the Gauss-Weingarten equations.  
Bobenko \cite{Bobenko94} gave an outline of eight classes of surfaces in the three-dimensional 
Euclidean space $\mathbb{E}^3$ in the formulation of the theory of solitons. 
They are minimal surfaces, surfaces of constant mean curvature, 
surfaces of constant positive Gaussian curvature, 
surfaces of constant negative Gaussian curvature, 
Bonnet surfaces, 
harmonic inverse mean curvature surfaces, 
Bianchi surfaces and Bianchi surfaces of positive curvature. 
For the investigation of these surfaces, $2\times 2$ matrices representation of quaternions is 
used to write their moving frames. 
Their moving frames are integrated by Sym's formula \cite{Sym85}.  

\textit{Quaternionic analysis} by Pedit and Pinkall \cite{PP98} is a technology to investigate surfaces in the Euclidean three- or four-space which are related to 
the soliton theory. 
In this theory,  
the Euclidean four-space $\mathbb{E}^4$ is modeled on the set of all quaternions $\mathbb{H}$. 
A quaternionic line trivial bundle with a complex structure over a Riemann surface
is associated with a conformal map from the Riemann surface to $\mathbb{E}^4$. 

We can assume that a quaternionic line trivial bundle associated with a \textit{constrained Willmore surface} in $\mathbb{E}^4$ 
equips a harmonic complex structure (\cite{Bohle10}).  
If a constrained Willmore surface is neither minimal nor super-conformal, then 
this complex structure defines a smooth family $\{\nabla^\mu\}_{\mu\in S^1}$ of flat connections on the line bundle. 
Then a \textit{holonomy spectral curve} of a constrained Willmore torus is defined by a smooth family of holonomies of $\{\nabla^\mu\}_{\mu\in S^1}$. 
The relation between a constrained Willmore torus and its holonomy spectral curve is 
discussed in detail in \cite{Bohle10}. 

If a conformal map from a torus to $\mathbb{E}^4$ is not a constrained Willmore torus, 
then the quaternionic line trivial bundle associated with the conformal map
is accompanied with a non-harmonic complex structure.  
For a conformal map, its \textit{Darboux transform} is defined. 
Then all Darboux transforms of a conformal map describe a \textit{multiplier spectral curve} of a conformal map (see \cite{BLPP12,BPP09,Bohle10}). 
The property of a conformal map from a torus to $\mathbb{E}^4$ is related to 
that of its multiplier spectral curve. 
Hence a Darboux transform of a conformal map which is not constrained Willmore 
is an important research subject. 

A \textit{harmonic inverse mean curvature surface} with non-constant mean curvature is a conformal map which is not constrained Willmore. 
We abbreviate a harmonic inverse mean curvature surface as a HIMC surface. 
A HIMC surface is introduced by Bobenko \cite{Bobenko94}. 
A transform of a HIMC surface is defined in \cite{Korotkin99}. 
All \textit{isothermic} HIMC surfaces 
are classified in terms of \textit{Painlev\'{e} transcendents} in \cite{BEK97} (see \cite{BE00}). 
We investigate a Darboux transform of a HIMC surface. 

For a conformal map $\mathfrak{f}\colon M\to\mathbb{H}$, 
a \textit{mean curvature sphere} of $\mathfrak{f}$ is defined by \eqref{CC}. 
We define a function $H\colon M\to\mathbb{H}$ by \eqref{mcs}.  
If $\mathfrak{f}(M)\subset\Img\mathbb{H}$ up to translations in $\mathbb{H}$, then $H$ is real-valued and exactly a mean curvature function of $\mathfrak{f}$ with the induced metric on $M$. 
Hence a HIMC surface is a surface in $\Img\mathbb{H}$ such that $1/H$ is harmonic. 
We call a conformal map $\mathfrak{f}\colon M\to\mathbb{H}$ such that $H^{-1}$ is harmonic a \textit{generalized harmonic inverse mean curvature surface} and 
abbreviate it as a GHIMC surface (Definition \ref{defGHIMC}). 
We advance the theory of GHIMC surfaces in a similar way to that of \textit{Willmore surfaces} in the four-sphere $S^4$ in \cite{BFLPP02}. 

A one-form $w$ is constructed from a mean curvature sphere by \eqref{Hopfw}. 
A constrained Willmore surface in $\mathbb{E}^4$ is a conformal map such that there exists a one-form $\eta$ with $w+\eta$ is closed (\cite{Bohle10}). The one-form $w$ is closed if and only if  
a conformal map is Willmore (\cite{BFLPP02}). 
On the other hand, we have the following. 
\begin{theorem}\label{cond}
Let $L$ be a line bundle of a conformal map with mean curvature sphere $S$ such that
$L_p\neq \{v_\infty\lambda\,|\,\lambda\in\mathbb{H}\}$ for every $p\in M$.  
Set $\mathfrak{f}:=\sigma(L)\colon M\to\mathbb{H}$. 
Define a map $H\colon M\to\mathbb{H}$ by \eqref{mcs} and a one-form $w$ by \eqref{Hopfw}. 
The map $\mathfrak{f}$ is a GHIMC surface if and only 
if 
there exists $n\in\mathbb{H}$ with $n\neq 1$ 
such that $H^{-1}\ast w\,H^{-1}+n\ast dH^{-1}$ is closed.
\end{theorem}
A \textit{backward B\"{a}cklund transform} of a Willmore surface in \cite{BFLPP02} and \cite{LP05} is 
a B\"{a}cklund transforms in \cite{LP08}. 
Because a backward B\"{a}cklund transform of a Willmore surface is a Willmore surface, 
a sequence of Willmore surfaces is obtained by repeating to take a backward B\"{a}cklund transform (see \cite{LP08}). 

We define a counterpart of a backward B\"{a}cklund transform of a Willmore surface in a GHIMC surface as follows. 
\begin{theorem}\label{Btrans}
Let $L$ be a line bundle of a conformal map with mean curvature sphere $S$ such that
$L_p\neq \{v_\infty\lambda\,|\,\lambda\in\mathbb{H}\}$ for every $p\in M$.  
We assume that $\mathfrak{f}:=\sigma(L)\colon M\to\mathbb{H}$ is a GHIMC surface. 
Define $H\colon M\to\mathbb{H}$ and $N\colon M\to S^2$ by \eqref{mcs}. 
Let $\mu\colon M\to\mathbb{H}$ be a map such that $d\mu=\ast\,dH^{-1}$ and 
$\mathfrak{h}$ a map such that 
defined by 
\begin{gather*}
\mathfrak{h}=\frac{1}{2}(-\mu+\mathfrak{f}-NH^{-1}).
\end{gather*}
Let $\widetilde{M}$ be the subset of $M$ where $\mathfrak{h}$ is defined. 
The quaternionic vector space $\mathfrak{H}$ 
spanned by $v_L$ and $v_L\overline{\mathfrak{h}}$ defines 
a B\"{a}cklund transform of $L$ on $\widetilde{M}$. 
\end{theorem}
We call the map $\overline{\mathfrak{h}}$ a backward B\"{a}cklund transform of a GHIMC surface $\mathfrak{f}$. 
It is unclear whether 
a backward B\"{a}cklund transform of a GHIMC surface is a GHIMC surface. 
Hence a sequence of GHIMC surfaces constructed by backward B\"{a}cklund transforms is more involved. 

A backward B\"{a}cklund transform of a GHIMC surface gives 
a Darboux transform of a GHIMC surface as follows. 
\begin{theorem}\label{DTGHIMC}
Let $L$ be a line bundle of a conformal map with mean curvature sphere $S$ such that
$L_p\neq \{v_\infty\lambda\,|\,\lambda\in\mathbb{H}\}$ for every $p\in M$.  
We assume that $\mathfrak{f}:=\sigma(L)\colon M\to\mathbb{H}$ is a GHIMC surface. 
Let $\mathfrak{h}\colon\widetilde{M}\to\mathbb{H}$ be a map such that 
$\overline{\mathfrak{h}}$ is a backward B\"{a}cklund transform of $\mathfrak{f}$ on $\widetilde{M}$. 
Then there exists 
a map $\lambda_\infty\colon\widetilde{M}\to\mathbb{H}$ such that 
$\widehat{\mathfrak{f}}:=\lambda_\infty\overline{\mathfrak{h}}^{-1}+\mathfrak{f}$ is a Darboux transform of $\mathfrak{f}$ on $\widetilde{M}$. 
\end{theorem}

An isothermic surface in $\mathbb{E}^4$ is a surface which has a curvature line coordinate around each 
nonumblic point. 
For an isothermic surface, a \textit{classical} Darboux transform is defined (\cite{CLP}). 
If an isothermic surface is a GHIMC surface or a HIMC surface, then we have the 
following.  
\begin{theorem}\label{DT}
A classical Darboux transform of an isothermic GHIMC surface is a GHIMC surface. 
A classical Darboux transform of an isothermic  HIMC surface is a HIMC surface. 
\end{theorem}

We see that if a conformal map $\mathfrak{f}\colon M\to\mathbb{H}$ is \textit{invariant under a subgroup $\mathcal{B}$} of Euclidean motions in $\mathbb{E}^4$, then there is a Darboux transform $\widehat{f}$ invariant under $\mathcal{B}$ (Corollary \ref{inv}). 
A surface of revolution is a surface invariant under group of rotations around an axis. 
All HIMC surfaces of revolution are classified by the solutions to the \textit{Painlev\'{e} III equation in trigonometric form \eqref{PIII}} in \cite{BEK97} (see Theorem \ref{cHIMC}).  
The formula in Theorem \ref{cHIMC} gives a HIMC surface in $\mathbb{E}^3\cong\Img\mathbb{H}$ invariant under 
rotations $\mathcal{R}$ around the $k$-axis in $\Img\mathbb{H}$. 
Then we obtain a transform of the solutions to the Painlev\'{e} III equation in trigonometric form as follows.
\begin{theorem}\label{tPIII}
Let $\phi(x)$ be a solution to the Painlev\'{e} III equation in trigonometric form \eqref{PIII} 
with $\phi^\prime(x)+2\sin(\phi(x))\neq 0$ and 
$\mathfrak{f}\colon \{x+yi\in\mathbb{C}\,|\,x>0\}\to \Img\mathbb{H}$ be a HIMC surface of revolution defined by Theorem \ref{cHIMC}. 
Assume that $\widehat{\mathfrak{f}}:=\lambda_\infty\lambda_L^{-1}+\mathfrak{f}$ is a classical Darboux transform of $\mathfrak{f}$ with $d\lambda_\infty+d\mathfrak{f}\,\lambda_L=0$ and  
$B\lambda_\infty=\lambda_\infty\circ\tau_B$ for every $B\in\mathcal{R}$.
Then $\widehat{\mathfrak{f}}$ is a HIMC surface of revolution invariant under $\mathcal{R}$ and  
the function $\widehat{\phi}(x)$ defined by \eqref{phi} for $\widehat{\mathfrak{f}}$
is a solution to the Painlev\'{e} III equation in trigonometric form \eqref{PIII}. 
\end{theorem}

\begin{ack}
This work is supported by JSPS KAKENHI Grant number 22540064. 
\end{ack}
\section{Conformal maps}
We recall a conformal map and a line bundle associated with a conformal map (\cite{BFLPP02}). 

Throughout this paper, we assume that $M$ is a simply-connected Riemann surface with complex structure $J^M$ and maps, vector bundles, sections are smooth. 
An orientation of $M$ is fixed so that, for any non-zero tangent vector $v$ of $M$, 
an ordered pair ($v$, $J^Mv$) is a positive basis. 
We denote by $TM$ the tangent bundle of $M$ and by $T^\ast M$ the 
cotangent bundle of $M$. 

Let $E$ be a vector space. 
We denote by $\Omega^p(E)$ 
the set of $p$-forms on $M$ with values in $E$. 
For a vector bundle $F$ over $M$, we denote by 
$\Omega^p(F)$ the set of $p$-forms on $M$ with values in $F$. 
For $\omega\in\Omega^1(F)$, we define $\ast\,\omega$ by $\ast\,\omega=\omega\circ J^M$. 
We denote by $\underline{E}$ the trivial bundle over $M$. 
Then $\Omega^p(E)$ is naturally identified with $\Omega^p(\underline{E})$.

We model $\mathbb{R}^4$ on the set of all quaternions $\mathbb{H}$ and 
$\mathbb{R}^3$ on the set of all purely imaginary quaternions $\Img\mathbb{H}$. 
For a quaternion $a$, we denote by $\overline{a}$ the quaternionic conjugate of $a$. 
Set $\real a:=(a+\overline{a})/2$ and $\Img a:=(a-\overline{a})/2$.  
Then, the inner product of $a$ and $b\in\mathbb{R}^4\cong \mathbb{H}$ is 
$\langle a,b\rangle=\real(\overline{a}b)=2^{-1}(\overline{a}b+\overline{b}a)$ and 
the norm of $a\in\mathbb{H}$ is $|a|:=(\overline{a}a)^{1/2}$. 
We denote by $\Img\mathbb{H}$ the set of all purely imaginary quaternions. 
If $a$, $b\in\Img\mathbb{H}$, then $ab$ is the vector product $a\times b$. 
Let $S^2$ be the sphere of radius one centered at the origin in $\Img\mathbb{H}$. 
Then, $S^2=\{a\in\Img\mathbb{H}\,|\,a^2=-1\}$. 
Hence, $S^2$ is the set of all square roots of $-1$ in $\Img\mathbb{H}$. 

Let $\omega\in\Omega^1(\mathbb{H})$. 
We define $\omega_N\in \Omega^1(\mathbb{H})$ and $\omega^N\in\Omega^1(\mathbb{H})$ by setting 
\begin{gather*}
\omega_N:=\frac{1}{2}(\omega-N\ast\,\omega),\enskip 
\omega^N:=\frac{1}{2}(\omega-\ast\,\omega\,N). 
\end{gather*}
Then, $\omega$ decomposes because
$\omega=\omega_N+\omega_{-N}=\omega^N+\omega^{-N}$. 
We see that $\ast\,\omega_N=N\,\omega_N$ and $\ast\,\omega^N=\omega\,N$. 
Clearly, $\omega=\omega_N$ if and only if $\omega_{-N}=0$. 
Similarly, $\omega=\omega^N$ if and only if $\omega^{-N}=0$. 
The quaternionic conjugation provides an identity
$\overline{\omega_N}=\overline{\omega}^{-N}$. 
It is known that $\omega^N\wedge\omega_{N}=0$ (\cite{BFLPP02}, Proposition 16). 

Let $E$ and $F$ be right quaternionic vector bundles over a manifold. 
We denote by $\Hom(E,F)$ the real vector bundle of quaternionic bundle homomorphism from $E$ to $F$, 
and $\End(E)$ the real vector bundle of quaternionic bundle endomorphisms of $E$. 
We denote the identity automorphism of $E$ by $\Id$. 
We call $J\in \Omega^0(\End(E))$ a \textit{complex structure} of $E$ if $J^2=-\Id$. 

We denote by $V$ the right quaternionic vector space of dimension two and 
by $\mathbb{P}(V)$ the quaternionic projective space of dimension one.  
We model the conformal four sphere $S^4$ on $\mathbb{P}(V)$.  
Let $\mathcal{T}$ be the tautological line bundle of $\mathbb{P}(V)$. 
We denote the trivial quaternionic vector bundle over $\mathbb{P}(V)$ of rank two by $\mathcal{V}$. 
The tangent bundle $T\mathbb{P}(V)$ of $\mathbb{P}(V)$ is identified with 
a bundle $\Hom(\mathcal{T},\mathcal{V}/\mathcal{T})$. 

We fix a basis $(v_0,v_\infty)$ of $V$. 
Let $V^\ast$ be the dual vector space of $V$. 
There exists uniquely a basis $(v_0^\ast,v_\infty^\ast)$ of $V^\ast$ such that 
\begin{gather*}
v_0^\ast(v_0)=1,\enskip v_0^\ast(v_\infty)=0,\enskip v^\ast_\infty(v_0)=0,\enskip v^\ast_\infty(v_\infty)=1.
\end{gather*}
We define a map $\sigma \colon V\setminus \{v_\infty\lambda\,|\,\lambda\in\mathbb{H}\}\to\mathbb{H}$ by
\begin{gather*}
\sigma(v)=(v_\infty^\ast(v))(v^\ast_0(v))^{-1}
\end{gather*}
for every $v\in V$.   
The map $\sigma$ induces a map from $\mathbb{P}(V)\setminus\{v_\infty\lambda\,|\,\lambda\in\mathbb{H}\}$ to $\mathbb{H}$. 
This map is a \textit{stereographic projection} of $S^4\cong\mathbb{P}(V)$ from 
$\{v_\infty\lambda\,|\,\lambda\in\mathbb{H}\}\in\mathbb{P}(V)$. 

Let $f\colon M\to \mathbb{P}(V)\cong S^4$ be a map and $L:=f^\ast \mathcal{T}$. 
Then $L$ is a quaternionic line subbundle of $\underline{V}$ such that $L_p=f(p)$ for every $p\in M$. 
Conversely, if $L$ is a line subbundle of $\underline{V}$, then there exists a map 
$f\colon M\to \mathbb{P}(V)$ such that $L=f^\ast \mathcal{T}$. 
We call $L$ the \textit{line bundle of a conformal map} $f$. 
We use the terminology of maps for line bundles.

Let  $\pi^{\underline{V}/L}\colon \underline{V}\to \underline{V}/L$ be the projection. 
We consider $v_0$ and $v_\infty$ as sections of $\underline{V}$. 
Define a flat quaternionic connection $\nabla$ on $\underline{V}$ by $\nabla v_0=\nabla v_\infty=0$.  
Then the differential map $df\colon TM\to T\mathbb{P}(V)$ is identified with 
a one-form $\delta:=\pi^{\underline{V}/L}\,\nabla|_{\Omega^0(L)}\in\Omega^1(\Hom(L,\underline{V}/L))$. 
A line bundle $L$ 
is called a line bundle of a conformal map if there exists a complex structure $S\in\Omega^0(\End(\underline{V}))$ such that 
\begin{gather}
SL=L,\enskip \nabla S|_L \in\Omega^1(L),\enskip  \ast\,\delta=\pi^{\underline{V}/L} S\,\nabla|_{\Omega^0(L)}=\delta S|_{\Omega^0(L)}.\label{CC}
\end{gather}
The complex structure $S$ is called the \textit{mean curvature sphere} or the 
\textit{conformal Gauss map} of $L$. 

Let  $L$ be a line bundle of a conformal map with 
mean curvature sphere $S$. 
We assume that $L_p\neq \{v_\infty\lambda\,|\,\lambda\in\mathbb{H}\}$ for every $p\in M$.  
We define a map $\mathfrak{f}\colon M\to\mathbb{H}$ by $\mathfrak{f}(p):=\sigma(L_p)$ for each $p\in M$. 
The section  $v_L:=v_0+v_\infty\mathfrak{f}$ of $L$ does not vanish anywhere. 
We see that 
$(v_\infty,v_L)$ is a frame of $\underline{V}$. 
We call $S$ the mean curvature sphere of $\mathfrak{f}$. 
There exist $H\colon M\to\mathbb{H}$, $N\colon M\to S^2$ and $R\colon M\to S^2$ such that  
\begin{gather}
\begin{gathered}
Sv_\infty=v_\infty N+v_L(-H),\enskip Sv_L=v_L(-R),\\
N^2=R^2=-1,\enskip 
RH=HN,\\
d\mathfrak{f}\,H=(dN)_N,\enskip 
H\,d\mathfrak{f}=(dR)^{-R}. 
\end{gathered}\label{mcs}
\end{gather}
Then $(df)_{-N}=(df)^R=0$. 

Conversely, let $\mathfrak{f}\colon M\to\mathbb{H}$ be a map with 
$(d\mathfrak{f})_{-N}=(d\mathfrak{f})^{R}=0$. 
Then a map $H\colon M\to\mathbb{H}$ is defined 
by \eqref{mcs} except at the branch points or $\mathfrak{f}$.
We define a line bundle $L$ by $L_p=\{(v_0+v_\infty\mathfrak{f}(p))\lambda\,|\,\lambda\in\mathbb{H}\}$ 
for each $p\in M$. Then 
$L$ is a line bundle of a conformal map 
with mean curvature sphere $S$ defined by \eqref{mcs}, except at the branch points of $\mathfrak{f}$.

Let $L$ be a line bundle of a conformal map with 
mean curvature sphere $S$ such that
$L_p\neq \{v_\infty\lambda\,|\,\lambda\in\mathbb{H}\}$ for every $p\in M$.
We induce a (singular) metric on $M$ by $\mathfrak{f}=\sigma(L)$. 
Then $\mathfrak{f}$ is a conformal map except at the branch points of $\mathfrak{f}$. 
The map $\mathfrak{f}$ is called a conformal map with left normal $N$ and right normal $R$. 
Let $\mathcal{H}$ be the mean curvature vector of $\mathfrak{f}$. Then 
$\mathcal{H}=-N\bar{H}=-\bar{H}R$. 
The image of $\mathfrak{f}$ is contained in $\Img\mathbb{H}$ up to translations in $\mathbb{H}$ 
if and only if $N=R$. If the image of $\mathfrak{f}$ is contained in $\Img\mathbb{H}$ up to translations in $\mathbb{H}$, then $H$ is the mean curvature function of $\mathfrak{f}$.  

We define $A$, $Q\in\Omega^1(\End(\underline{V}))$ by 
\begin{gather}
2A=S\,dS+\ast \,dS,\enskip 2Q=S\,dS-\ast \,dS.\label{Hopf}
\end{gather}
The one-forms $A$ and $Q$ are called the \textit{Hopf fields} of $L$. 
We have 
\begin{gather}
\begin{gathered}
4Av_\infty=v_L(-\ast w),\enskip 4Av_L=v_L2R(dR)_{-R},\\
4Qv_\infty=v_\infty 2N(dN)_{-N}+v_L\ast(2\,dH-w), \enskip 4Qv_L=0,\\ 
w=dH+H\ast d\mathfrak{f}\,H+R\ast dH-H\ast dN. 
\end{gathered}\label{Hopfw}
\end{gather}

For $C\in\End(V)$, we set $\langle C\rangle:=\frac{1}{8}\tr_{\mathbb{R}}C$. 
Then an indefinite scalar product $\langle\enskip,\enskip\rangle$ of $\End(V)$ is defined by 
setting $\langle C_1,C_2\rangle:=\langle C_1C_2\rangle$ for 
$C_1$, $C_2\in\End(V)$. 
The functional 
\begin{gather*}
W(L):=\frac{1}{\pi}\int_M\langle A\wedge\ast A\rangle
\end{gather*}
is called the \textit{Willmore functional} of $L$. 
A conformal map whose line bundle is 
a critical line bundle of $W$ is called a \textit{Willmore surface}. 
A critical line bundle is called a \textit{Willmore line bundle}. 
It is known that a line bundle $L$ of a conformal map 
with $L_p\neq\{v_\infty\lambda\,|\,\lambda\in\mathbb{H}\}$ for every $p\in M$ is Willmore if and only if $dw=0$ (\cite{BFLPP02}, Proposition 15). 

The following lemma is proved in \cite{BFLPP02} implicitly. 
\begin{lemma}\label{eta}
We have 
$(2\,dH-w)^{-N}=0$. 
\end{lemma}
\begin{proof}
We have 
\begin{gather*}
v_\infty 2N(dN)_{-N}+v_L\ast(2\,dH-w)=4Qv_\infty=-\ast\,4QSv_\infty\\
=-\ast\,4Q(v_\infty N+v_L(-H))=(v_\infty(-2(dN)_{-N})+v_L(2\,dH-w))N\\
=v_\infty 2N(dN)_{-N}+v_L(2\,dH-w)N.
\end{gather*}
Hence, $\ast\,(2\,dH-w)=(2\,dH-w)N$.  Then $(2\,dH-w)^{-N}=0$. 
\end{proof}

\section{Transforms}
We recall B\"{a}cklund transforms and Darboux transforms of a conformal map in \cite{LP05, BLPP12}. 

Let $L$ be a line bundle of a conformal map with mean curvature sphere $S$ such that 
$L_p\neq \{v_\infty\lambda\,|\,\lambda\in\mathbb{H}\}$ for every $p\in M$.  
Set $\mathfrak{f}:=\sigma(L)\colon M\to\mathbb{H}$ and $v_L:=v_0+v_\infty\mathfrak{f}$. 
Define $N\colon M\to S^2$, $R\colon M\to S^2$ and $H\colon M\to\mathbb{H}$ by \eqref{mcs}.  
We have a splitting $\underline{V}=L\oplus(\underline{V}/L)$. 
Let $\pi^L\colon \underline{V}\to L$ be the projection. 
We decompose $\nabla|_{\Omega^0(L)}$ by this splitting as 
\begin{gather*}
\nabla|_{\Omega^0(L)}=\nabla^L+\delta,\enskip 
\nabla^L=\pi^L\nabla|_{\Omega^0(L)}. 
\end{gather*}
We see that $\nabla^L$ is a connection of $L$. 
Because $\nabla^Lv_L=\pi^L\nabla v_L=\pi^Lv_\infty\,d\mathfrak{f}=0$, 
the section $v_L$ is a nowhere-vanishing parallel section of $\nabla^L$. 
Hence $\nabla^L$ is a flat connection.  
Set $D^L:=2^{-1}(\nabla^L+S\ast\nabla^L)$. 
The operator $D^L$ is a \textit{quaternionic holomorphic structure} of $L$ (see \cite{LP05}). 
Because $\nabla^L v_L=0$, we have $D^Lv_L=0$. 

Assume that $\lambda_0\colon M\to\mathbb{H}$ and $\lambda_1\colon M\to\mathbb{H}$ be maps such that 
$v_L\lambda_0$ and $v_L\lambda_1$ are linearly independent, 
$D^L(v_L\lambda_0)=D^L(v_L\lambda_1)=0$ 
and $|\lambda_0(p)|+|\lambda_1(p)|\neq 0$ for any $p\in M$. 
Let $\mathfrak{H}$ be the two dimensional vector space spanned by 
$v_L\lambda_0$ and $v_L\lambda_1$. Then $\underline{\mathfrak{H}}$ is a trivial bundle over $M$ of rank two. 
Let $\ev\colon \underline{\mathfrak{H}}\to L$ be the evaluation map defined by 
\begin{gather*}
\ev(p,v_L\lambda_0\mu_0+v_L\lambda_1\mu_1)=(p,(v_L(\lambda_0\mu_0+ \lambda_1\mu_1))(p)). 
\end{gather*}
The map $\ev\colon \underline{\mathfrak{H}}\to L$ is a surjective bundle homomorphism. 
Hence we have an injective bundle homomorphism $\ev^\ast\colon L^{-1}\to \underline{\mathfrak{H}^\ast}$. 
Let $L^\ast$ and $\underline{\mathfrak{H}^\ast}$ be the dual bundles of $L$ and $\underline{\mathfrak{H}}$ respectively. 
Then, we have a line bundle $\ev^\ast(L^\ast)\subset \underline{\mathfrak{H}^\ast}$ of a conformal map.  
The line bundle $\ev^\ast(L^\ast)$ is called the \textit{B\"{a}cklund transform} of $L$ with respect to $\mathfrak{H}$. 

\begin{lemma}\label{cBtrans}
A non-constant map $\lambda\colon M\to\mathbb{H}$ satisfies the equation 
$(d\lambda)_R=0$ if and only if the subspace spanned by 
$v_L$ and $v_L\lambda$ defines a B\"{a}cklund transform of $L$. 
\end{lemma}
\begin{proof}
Recall that $D^Lv_L=0$ and $v_L$ is nowhere-vanishing. 
For a non-constant map $\lambda\colon M\to\mathbb{H}$, we have 
\begin{gather*}
2D^Lv_L\lambda=\nabla^L(v_L\lambda)+S\ast\nabla^L(v_L\lambda)=v_L\,2(d\lambda)_R. 
\end{gather*}
Hence $D^L(v_L\lambda)=0$ if and only if $(d\lambda)_{R}=0$.  
Thus $v_L$ and $v_L\lambda$ defines a B\"{a}cklund transform if and only if $\lambda$ is not constant and $(d\lambda)_{R}=0$.  
\end{proof}
\begin{lemma}\label{BDtrans}
Let $\lambda_\infty\colon M\to\mathbb{H}$ and 
$\lambda_L\colon M\to\mathbb{H}$ be maps satisfying the equation 
$d\lambda_\infty+d\mathfrak{f}\,\lambda_L=0$. 
Then
a subspace of $\Omega^0(L)$ spanned by $v_L$ and $v_L\lambda_L$ 
defines a B\"{a}cklund transform of $L$. 
\end{lemma}
\begin{proof}
We have $0=d(d\lambda_\infty)=d(-d\mathfrak{f}\,\lambda_L)=d\mathfrak{f}\wedge d\lambda_L$. 
Hence $(d\lambda_L)_R=0$. 
By Lemma \ref{Btrans}, 
a subspace of $\Omega^0(L)$ spanned by $v_L$ and $v_L\lambda_L$ 
defines a B\"{a}cklund transform of $L$. 
\end{proof}
For $\phi\in\Omega^0(\underline{V}/L)$, we denote by $\tilde{\phi}\in\Omega^0(\underline{V})$ 
a lift of $\phi$, that is $\pi^{\underline{V}/L}\tilde{\phi}=\phi$. 
Set 
\begin{gather*}
D^{\underline{V}/L}\phi:=\pi^{\underline{V}/L}\left(\frac{1}{2}(\nabla\tilde{\phi}+S\ast\,\nabla\tilde{\phi})\right).
\end{gather*}
The operator $D^{\underline{V}/L}$ is a quaternionic holomorphic structure of $\underline{V}/L$. 
It is known that if $D^{\underline{V}/L}\phi=0$, then there exists a unique lift $\widehat{\phi}$ such that 
$\pi^{\underline{V}/L}\nabla\widehat{\phi}=0$. 
Then, a line bundle 
$\widehat{L}$ with $\widehat{L}_p=\{\widehat{\phi}(p)\lambda\,|\,\lambda\in \mathbb{H}\}$ for every $p\in M$ away from zeros of $\widehat{\phi}$   
is a line bundle of a conformal map.  
Set $\widehat{\delta}:=\pi^{\underline{V}/\widehat{L}}\nabla$. 
Then $S\widehat{L}=\widehat{L}$ and 
$\ast\,\widehat{\delta}=S\,\widehat{\delta}$. 
The line bundle $\widehat{L}$ of a conformal map is called 
a \textit{Darboux transform} of $L$ with respect to $\phi$. 
We call the map $\widehat{\mathfrak{f}}=\sigma(\widehat{L})$ a \textit{Darboux transform} of 
$\mathfrak{f}$ with respect to $\phi$.

A Darboux transform $\widehat{L}$ is called \textit{classical} if $L_p\neq \widehat{L}_p$ for each $p\in M$ and $S$ is the mean curvature sphere of $\widehat{L}$ (\cite{CLP}). 
Let $\widehat{S}$ be the mean curvature sphere of $\widehat{L}$ and set $\widehat{\mathfrak{f}}:=\sigma(\widehat{L})$. 
If $\widehat{L}$ is a classical Darboux transform, then 
$\mathfrak{f}$ and $\widehat{\mathfrak{f}}$ share conformal curvature line parametrizations. The conformal maps $\mathfrak{f}$ and $\widehat{\mathfrak{f}}$ are 
called \textit{isothermic surfaces} (see \cite{CLP}).  

The equation $d\lambda_\infty+d\mathfrak{f}\,\lambda_L=0$ is related with a Darboux transform as follows. 
\begin{lemma}\label{Dtrans}
If maps $\lambda_\infty\colon M\to\mathbb{H}$ and 
$\lambda_L\colon M\to\mathbb{H}$ satisfy the equation 
$d\lambda_\infty+d\mathfrak{f}\,\lambda_L=0$, then
$\widehat{\mathfrak{f}}=\lambda_\infty\lambda_L^{-1}+\mathfrak{f}$ is a Darboux transform of $\mathfrak{f}$ with respect to $v_\infty\lambda_\infty$ and 
$\overline{\mathfrak{f}}$ is a Darboux transform of $\overline{\widehat{\mathfrak{f}}}$ 
with respect to $v_\infty\overline{\lambda_L^{-1}}$. 
\end{lemma}
\begin{proof}
We identify $\underline{V}/L$ with $\underline{\ker v_0^\ast}=\underline{\{v_\infty\lambda\,|\,\lambda\in\mathbb{H}\}}$.
For $v_\infty\lambda_\infty\in\Omega^0(\underline{\ker v_0^\ast})\cong\Omega^0(\underline{V}/L)$ with $\lambda_\infty\colon M\to\mathbb{H}$, we have 
\begin{gather*}
2D^{\underline{V}/L}v_\infty\lambda_\infty=\pi^{\underline{V}/L}(v_\infty\,d\lambda_\infty+Sv_\infty\,\ast\,d\lambda_\infty)\\
=\pi^{\underline{V}/L}(v_\infty(d\lambda_\infty+N\ast d\lambda_\infty))=v_\infty2(d\lambda_\infty)_{-N}. 
\end{gather*}
Hence $D^{\underline{V}/L}v_\infty\lambda_\infty=0$ if and only if $(d\lambda_\infty)_{-N}=0$. 

A lift of $\phi=v_\infty\lambda_\infty$ is written as $\tilde{\phi}=v_\infty\lambda_\infty+v_L\lambda_L$ for $\lambda_L\colon M\to\mathbb{H}$. 
Because $\sigma(L_p)=\mathfrak{f}(p)$ for every $p\in M$, we have
\begin{gather*}
\pi^{V/L}\nabla\tilde{\phi}=\pi^{V/L}(v_\infty\,d\lambda_\infty+v_\infty\,d\mathfrak{f}\,\lambda_L+v_L\,d\lambda_L)
=v_\infty(d\lambda_\infty+d\mathfrak{f}\,\lambda_L).
\end{gather*} 
Hence $\tilde{\phi}=v_\infty\lambda_\infty+v_L\lambda_L$ satisfies the equation $\pi^{V/L}\nabla\tilde{\phi}=0$ if and only if $d\lambda_\infty+d\mathfrak{f}\,\lambda_L=0$. 
Thus
if maps $\lambda_\infty\colon M\to\mathbb{H}$ and $\lambda_L\colon M\to\mathbb{H}$ satisfy the equation 
$d\lambda_\infty+d\mathfrak{f}\,\lambda_L=0$,
then a line bundle 
$\widehat{L}$ with $\widehat{L}_p=\{\tilde{\phi}(p)\lambda\,|\,\lambda\in \mathbb{H}\}$ for every $p\in M$ away from zero of $\tilde{\phi}$ is a Darboux transform of $L$ with respect to $v_\infty\lambda_\infty$. 
Because 
\begin{gather*}
\tilde{\phi}=v_\infty\lambda_\infty+v_L\lambda_L=v_0\lambda_L+v_\infty(\lambda_\infty+\mathfrak{f}\lambda_L),
\end{gather*}
we have $\sigma(\widehat{L})=(\lambda_\infty+\mathfrak{f}\,\lambda_L)\lambda_L^{-1}=\lambda_\infty\lambda_L^{-1}+\mathfrak{f}=\widehat{\mathfrak{f}}$. 

We have $d\widehat{\mathfrak{f}}=\lambda_\infty\,d\lambda_L^{-1}$. 
Hence $d\overline{\lambda_L^{-1}}+d\overline{\widehat{\mathfrak{f}}}\,(-\overline{\lambda_\infty^{-1}})=0$. 
Then $\overline{\lambda_L^{-1}}(-\overline{\lambda_\infty})+\overline{\widehat{\mathfrak{f}}}=\overline{\mathfrak{f}}$ is a Darboux transform of $\overline{\widehat{\mathfrak{f}}}$ with respect to $v_\infty\overline{\lambda_L^{-1}}$. 
\end{proof}

We have the following by straightforward calculation. 
\begin{lemma}\label{DTNR}
Define $v_{\widehat{L}}$, $T_\infty$, $T_L$, $\widehat{N}$, $\widehat{R}$ and $\widehat{H}$ by 
\begin{gather}
\begin{gathered}
v_{\widehat{L}}=v_0+v_\infty\widehat{\mathfrak{f}}=v_\infty T_\infty+v_LT_L,\\
\widehat{S}v_\infty=v_\infty \widehat{N}+v_{\widehat{L}}(-\widehat{H}),\enskip \widehat{S}v_{\widehat{L}}=v_{\widehat{L}}(-\widehat{R}),\\
\widehat{N}^2=\widehat{R}^2=-1,\enskip 
\widehat{R}\widehat{H}=\widehat{H}\widehat{N},\\
d\widehat{\mathfrak{f}}\,\widehat{H}=(d\widehat{N})_{\widehat{N}},\enskip 
\widehat{H}\,d\widehat{\mathfrak{f}}=(d\widehat{R})^{-\widehat{R}}. 
\end{gathered}
\label{DTmcs}
\end{gather}
Then, 
\begin{gather*}
T_L=1,\enskip T_\infty=\widehat{\mathfrak{f}}-\mathfrak{f},\\
N T_\infty + T_\infty H T_\infty + T_\infty R=0,\\
\widehat{N}-N=(\widehat{\mathfrak{f}}-\mathfrak{f})H,\enskip \widehat{R}-R=\widehat{H}(\widehat{\mathfrak{f}}-\mathfrak{f}). 
\end{gather*}
The map $\mathfrak{f}$ is an isothermic surface and $\widehat{\mathfrak{f}}$ is a classical Darboux transform 
if and only if $\widehat{R}=-T_\infty^{-1}NT_\infty$. 
If a map $\mathfrak{f}$ is an isothermic surface and $\widehat{\mathfrak{f}}$ is a classical Darboux transform then $\widehat{H}=H$. 
\end{lemma}
\begin{proof}
Because
\begin{gather*}
v_\infty T_\infty+v_LT_L=v_\infty T_\infty+(v_0+v_\infty\mathfrak{f})T_L
=v_0T_L+v_\infty(T_\infty+\mathfrak{f}T_L),
\end{gather*}
we have $T_L=1$ and $T_\infty=\widehat{\mathfrak{f}}-\mathfrak{f}$. 

Because
\begin{gather*}
S\widehat{L}=\widehat{L},\\
Sv_{\widehat{L}}=S(v_\infty T_\infty+v_L)
=(v_\infty N+v_L(-H)) T_\infty +v_L(-R)\\
=v_\infty N T_\infty +v_L(-H T_\infty -R)\\
=v_\infty N T_\infty +(-v_\infty T_\infty +v_{\widehat{L}})(-H T_\infty -R)\\
=v_\infty (N T_\infty + T_\infty H T_\infty + T_\infty R)
+v_{\widehat{L}}(-HT_\infty -R), 
\end{gather*}
we have $N T_\infty + T_\infty H T_\infty + T_\infty R=0$. 
Because 
\begin{gather*}
-HT_\infty -R=-T_\infty^{-1}(T_\infty HT_\infty +T_\infty R)=T_\infty^{-1}NT_\infty, 
\end{gather*}
we have  
$Sv_{\widehat{L}}=v_{\widehat{L}}T_\infty^{-1}NT_\infty$.  

Because $\ast\,\widehat{\delta}=\widehat{S}\widehat{\delta}=S\widehat{\delta}$, we have 
\begin{gather*}
\pi^{\underline{V}/\widehat{L}}\widehat{S}v_\infty=\pi^{\underline{V}/\widehat{L}}Sv_\infty,\\
Sv_\infty=v_\infty N+v_L(-H)
=v_\infty N+(-v_\infty T_\infty +v_{\widehat{L}})(-H)\\
=v_\infty(N+T_\infty H)+v_{\widehat{L}}(-H),\\
\widehat{S}v_\infty=v_\infty \widehat{N}+v_{\widehat{L}}(-\widehat{H}).
\end{gather*}
Hence $\widehat{N}-N=(\widehat{\mathfrak{f}}-\mathfrak{f})H$. 

By Lemma \ref{Dtrans}, a conformal map $\overline{\mathfrak{f}}$ is a Darboux transform of $\overline{\widehat{\mathfrak{f}}}$. 
Because $d\overline{\mathfrak{f}}(-\overline{H})=(dR)_R$ and 
$d\overline{\widehat{\mathfrak{f}}}\,(-\overline{\widehat{H}})=(d\widehat{R})_{\widehat{R}}$, we have $R-\widehat{R}=(\overline{\mathfrak{f}}-\overline{\widehat{\mathfrak{f}}})(-\overline{\widehat{H}})$. 
Then $\widehat{R}-R=\widehat{H}(\widehat{\mathfrak{f}}-\mathfrak{f})$. 

We assume that $\mathfrak{f}$ is an isothermic surface and a Darboux transform $\widehat{L}$ 
of $L$ is classical. 

We have 
\begin{gather*}
\widehat{\delta}Sv_{\widehat{L}}=\widehat{\delta}v_{\widehat{L}}T_\infty^{-1}NT_\infty
=v_\infty\,d\widehat{f}\,T_\infty^{-1}NT_\infty,\\
\widehat{\delta}\widehat{S}v_{\widehat{L}}=\widehat{\delta}v_{\widehat{L}}(-\widehat{R})=v_\infty\,d\widehat{f}\,(-\widehat{R}). 
\end{gather*}
Hence $\widehat{R}=-T_\infty^{-1}NT_\infty$. 

Conversely, if $\widehat{R}=-T_\infty^{-1}NT_\infty$, then $\ast\,\widehat{\delta}=\widehat{\delta}\widehat{S}=\widehat{\delta}S$. 
Because $\widehat{L}$ is a Darboux transform of $L$, we have 
$\widehat{S}\widehat{\delta}=S\widehat{\delta}$. Hence 
$\widehat{L}$ is a classical Darboux transform of $L$. 

If $\mathfrak{f}$ is an isothermic surface and $\widehat{\mathfrak{f}}$ is a classical Darboux transform, 
then $\mathfrak{f}$ is a classical Darboux transform of $\widehat{\mathfrak{f}}$. 
Hence $N-\widehat{N}=(\mathfrak{f}-\widehat{\mathfrak{f}})\widehat{H}$. 
Then $\widehat{H}=H$. 
\end{proof}
A pair of two conformal maps $\mathfrak{f}$ and $\mathfrak{g}$ is called a \textit{Christoffel pair} in 
\cite{HP97} if $d\overline{\mathfrak{f}}\wedge d\mathfrak{g}=d\mathfrak{g}\wedge d\overline{\mathfrak{f}}=0$. 
It is known that if a pair $\mathfrak{f}$ and $\mathfrak{g}$ is a Christoffel pair, 
then maps $\mathfrak{f}$ and $\mathfrak{g}$ are isothermic. 
If a map $\mathfrak{f}$ is isothermic, then 
there exists a map $\mathfrak{g}$ such that a pair of $\mathfrak{f}$ and $\mathfrak{g}$ is a Christoffel pair (see \cite{HP97}). 

\begin{corollary}\label{isothermic}
We assume that conformal maps $\lambda_\infty\colon M\to\mathbb{H}$, 
$\lambda_L\colon M\to\mathbb{H}$, $\mathfrak{f}\colon M\to\mathbb{H}$ and $\mathfrak{g}\colon M\to\mathbb{H}$ satisfy the equation 
$d\lambda_\infty+d\mathfrak{f}\,\lambda_L=0$ and $d\lambda_L+d\mathfrak{g}\,\lambda_\infty=0$. 
Set $\widehat{\mathfrak{f}}:=\lambda_\infty\lambda_L^{-1}+\mathfrak{f}$ and 
$\widehat{\mathfrak{g}}:=\lambda_L\lambda_\infty^{-1}+\mathfrak{g}$. 
If $d\lambda_\infty\wedge d\lambda_L^{-1}=d\lambda_L\wedge d\lambda_\infty^{-1}=0$, then  
$d\mathfrak{f}\wedge d\mathfrak{g}=d\mathfrak{g}\wedge d\mathfrak{f}=0$ and 
$d\widehat{\mathfrak{f}}\wedge d\widehat{\mathfrak{g}}=d\widehat{\mathfrak{g}}\wedge d\widehat{\mathfrak{f}}=0$. 
The conformal maps $\widehat{\mathfrak{f}}$ and $\widehat{\mathfrak{g}}$ are 
classical Darboux transforms of $\mathfrak{f}$ and $\mathfrak{g}$ respectively. 
\end{corollary}
\begin{proof}
We have 
\begin{gather*}
d\mathfrak{f}\wedge d\mathfrak{g}=d\lambda_\infty\,\lambda_L^{-1}\wedge d\lambda_L\,\lambda_\infty^{-1}=-d\lambda_\infty\wedge d\lambda_L^{-1}\,\lambda_L\lambda_\infty^{-1}=0,\\
d\mathfrak{g}\wedge d\mathfrak{f}=d\lambda_L\,\lambda_\infty^{-1}\wedge d\lambda_\infty\,\lambda_L^{-1}=-d\lambda_L\wedge d\lambda_\infty^{-1}\,\lambda_\infty\lambda_L^{-1}=0.
\end{gather*}
Because $d\widehat{\mathfrak{f}}=\lambda_\infty\,d\lambda_L^{-1}$ and 
$d\widehat{\mathfrak{g}}=\lambda_L\,d\lambda_\infty^{-1}$, we have 
\begin{gather*}
d\widehat{\mathfrak{f}}\wedge d\widehat{\mathfrak{g}}=\lambda_\infty\,d\lambda_L^{-1}\wedge \lambda_L\,d\lambda_\infty^{-1}
=-\lambda_\infty\lambda_L^{-1}\,d\lambda_L\wedge d\lambda_\infty^{-1}=0,\\
d\widehat{\mathfrak{g}}\wedge d\widehat{\mathfrak{f}}=\lambda_L\,d\lambda_\infty^{-1}\wedge \lambda_\infty\,d\lambda_L^{-1}=-\lambda_L\lambda_\infty^{-1}\,d\lambda_\infty\wedge d\lambda_L^{-1}=0.
\end{gather*}
We assume that $(d\mathfrak{f})_{-N}=0$ and $(d\widehat{\mathfrak{f}})^{\widehat{R}}=0$. 
Because $d\mathfrak{f}=d\lambda_\infty\,\lambda_L^{-1}$, $d\widehat{\mathfrak{f}}=\lambda_\infty\,d\lambda_L^{-1}$, and 
$d\widehat{\mathfrak{f}}\wedge\lambda_L\,d\lambda_\infty^{-1}=-d\widehat{\mathfrak{f}}\wedge\lambda_L\lambda_\infty^{-1}\,d\lambda_\infty\,\lambda_\infty^{-1}=0$, we have 
$\widehat{R}=-\lambda_L\lambda_\infty^{-1}N\lambda_\infty\lambda_L^{-1}=-T_\infty^{-1} N T_\infty$. 
Hence $\widehat{\mathfrak{f}}$ is a classical Darboux transform of $\mathfrak{f}$. 
\end{proof}
We see that a pair of an isothermic surface in $\mathbb{E}^3$ and its classical Darboux transform forms 
a Bonnet pair in \cite{KPP98} up to similarities in $\mathbb{E}^3$. 

For an isothermic surface in $\mathbb{E}^3$, we have the following. 
\begin{corollary}\label{HIMC}
We assume that $\widehat{L}$ is a classical Darboux transform of a line bundle $L$ of 
an isothermic surface. 
The image of $\mathfrak{f}$ is contained in $\Img\mathbb{H}$ up to translations in $\mathbb{H}$ if and only if  
the image of $\widehat{\mathfrak{f}}$ is contained in $\Img\mathbb{H}$ up to translations in $\mathbb{H}$. 
\end{corollary}
\begin{proof}
Assume that the image of $\mathfrak{f}$ is contained in $\Img\mathbb{H}$ up to translations in $\mathbb{H}$. 
Then $N=R$ and $H$ is real-valued. 
By Lemma \ref{DTNR}, we have $\widehat{H}=H$. Hence $\widehat{H}$ is real-valued. 
Then 
\begin{gather*}
\widehat{N}=N+(\widehat{\mathfrak{f}}-\mathfrak{f})H,\\
\widehat{R}=R+\widehat{H}(\widehat{\mathfrak{f}}-\mathfrak{f})
=N+(\widehat{\mathfrak{f}}-\mathfrak{f})H=\widehat{N}. 
\end{gather*}
Then the image of $\widehat{\mathfrak{f}}$ is contained in $\Img\mathbb{H}$ up to translations in $\mathbb{H}$. 

The converse is trivial. 
\end{proof}

Let $B\colon \mathbb{H}\to\mathbb{H}$ be a Euclidean motion of $\mathbb{E}^4\cong\mathbb{H}$ fixing the origin. 
Then there exists $r$, $s\in\mathbb{H}$ with $|r|=|s|=1$ such that 
$Ba=ras^{-1}$ for any $a\in\mathbb{H}$. 
If $r=s$, then $B$ is a Euclidean motion of $\mathbb{E}^4\cong\mathbb{H}$ fixing $\mathbb{R}\subset\mathbb{H}$ and $\Img\mathbb{H}$. 
\begin{lemma}\label{EM}
Let $\mathfrak{f}\colon M\to\mathbb{H}$ be a conformal map 
and $\widehat{\mathfrak{f}}=\lambda_\infty\lambda_L^{-1}+\mathfrak{f}$ a Darboux transform of $\mathfrak{f}$ with respect to $v_\infty\lambda_\infty$ such that $d\lambda_\infty+d\mathfrak{f}\,\lambda_L=0$. 
Then $B\widehat{\mathfrak{f}}$ is a Darboux transform of $B\mathfrak{f}$ with respect to $v_\infty B\lambda_\infty$. 
\end{lemma}
\begin{proof}
We have 
\begin{gather*}
B\widehat{\mathfrak{f}}=B\lambda_\infty\,s\lambda_L^{-1}s^{-1}+B\mathfrak{f}.
\end{gather*}
Because $d\lambda_\infty+d\mathfrak{f}\,\lambda_L=0$, we have 
\begin{gather*}
d(B\lambda_\infty) +d(B\mathfrak{f})\,s\lambda_Ls^{-1}=0. 
\end{gather*}
Hence $B\widehat{\mathfrak{f}}$ is a Darboux transform of $B\mathfrak{f}$ with respect to $v_\infty B\lambda_\infty$. 
\end{proof}

For  a subgroup $\mathcal{B}$ of Euclidean motions of $\mathbb{E}^n$ $(n=3,4)$, 
we say that a map $\mathfrak{f}\colon M\to\mathbb{E}^n$ is \textit{invariant under $\mathcal{B}$} if 
there exists a subgroup $\{\tau_B\}_{B\in\mathcal{B}}$ of holomorphic transformations of $M$ such that 
$B\mathfrak{f}=\mathfrak{f}\circ\tau_B$ for any $B\in\mathcal{B}$. 

\begin{corollary}\label{inv}
Let $\mathcal{B}$ be a subgroup of Euclidean motions of $\mathbb{E}^4\cong\mathbb{H}$ such that 
each element of $\mathcal{B}$ fixes the origin of $\mathbb{H}$ 
and $\mathfrak{f}\colon M\to\mathbb{H}$ be a conformal map invariant under $\mathcal{B}$. 
Let $\tau_B\colon M\to M$ be a holomorphic transformations of $M$ such that $B\mathfrak{f}=\mathfrak{f}\circ \tau_B$ for each $B\in\mathcal{B}$. 
Assume that  $\lambda_\infty\colon M\to\mathbb{H}$ and $\lambda_L\colon M\to\mathbb{H}$ 
are maps such that $d\lambda_\infty+d\mathfrak{f}\,\lambda_L=0$ and $B\lambda_\infty=\lambda_\infty\circ \tau_B$ for every $B\in \mathcal{B}$.  
Set $\widehat{\mathfrak{f}}:=\lambda_\infty\lambda_L^{-1}+\mathfrak{f}$. 
Then $\widehat{\mathfrak{f}}$ is a Darboux transform of $\mathfrak{f}$ invariant under $\mathcal{B}$. 
\end{corollary}
\begin{proof}
Let $B\in\mathcal{B}$. Then there exist $r$, $s\in\mathbb{H}$ with $|r|=|s|=1$ such that $Ba=ras^{-1}$ for any $a\in \mathbb{H}$. 
We have 
\begin{gather*}
B(d\lambda_\infty+d\mathfrak{f}\,\lambda_L)=d(B\lambda_\infty)+d(B\mathfrak{f})\,s\lambda_Ls^{-1}=0,\\
\tau_B^\ast(d\lambda_\infty+d\mathfrak{f}\,\lambda_L)
=d(\lambda_\infty\circ\tau_B)+(d(\mathfrak{f}\circ\tau_B))\,\lambda_L\circ\tau_B\\
=d(B\lambda_\infty)+(d(B\mathfrak{f}))\,\lambda_L\circ\tau_B=0. 
\end{gather*}
Hence $\lambda_L\circ\tau_B=s\lambda_Ls^{-1}$. 
By the definition, the map $\widehat{\mathfrak{f}}$ is a Darboux transform of $\mathfrak{f}$ 
with respect to $v_\infty\lambda_\infty$. 
We have 
\begin{gather*}
B\widehat{\mathfrak{f}}=B(\lambda_\infty\lambda_L^{-1}+\mathfrak{f})
=(B\lambda_\infty)s\lambda_L^{-1}s^{-1}+B\mathfrak{f}\\
=(\lambda_\infty\circ\tau_B)(\lambda_L\circ\tau_B)^{-1}+\mathfrak{f}\circ\tau_B
=\widehat{\mathfrak{f}}\circ \tau_B. 
\end{gather*}
for each $B\in\mathcal{B}$. 
Hence $\widehat{\mathfrak{f}}$ is invariant under $B$. 
\end{proof}
\section{GHIMC surfaces}
A conformal map $\mathfrak{f}\colon M\to\Img\mathbb{H}$ is called a \textit{harmonic inverse mean curvature surface} if 
$d\ast dH^{-1}=0$ (\cite{Bobenko94}). 
We introduce the notion of a generalized harmonic inverse mean curvature surface and define its transform. 

Let $L$ be a line bundle of a conformal map with mean curvature sphere $S$ such that
$L_p\neq \{v_\infty\lambda\,|\,\lambda\in\mathbb{H}\}$ for every $p\in M$.  
Set $\mathfrak{f}:=\sigma(L)\colon M\to\mathbb{H}$. 
Define $H\colon M\to\mathbb{H}$ by \eqref{mcs}. 
\begin{definition}\label{defGHIMC}
We call the map $\mathfrak{f}$ a \textit{generalized harmonic inverse mean curvature surface} if 
$d\ast dH^{-1}=0$. 
\end{definition}
We abbreviate a harmonic inverse mean curvature surface as a HIMC surface and 
a generalized harmonic inverse mean curvature surface as a GHIMC surface. 
A GHIMC surface $\mathfrak{f}$ whose image is contained in $\Img\mathbb{H}$ up to translations in $\mathbb{H}$ 
is exactly a HIMC surface. 

\begin{proof}[Proof of Theorem \ref{DT}]
Let $L$ be a line bundle of a conformal map with mean curvature sphere $S$ such that
$L_p\neq \{v_\infty\lambda\,|\,\lambda\in\mathbb{H}\}$ for every $p\in M$ and 
$\widehat{L}$ be a Darboux transform of $L$. 
Set $\mathfrak{f}:=\sigma(L)$ and $\widehat{\mathfrak{f}}:=\sigma(\widehat{L})$. 
Define $H\colon M\to\mathbb{H}$ by \eqref{mcs} and $\widehat{H}\colon M\to\mathbb{H}$ by \eqref{DTmcs}. 
By Lemma \ref{DTNR}, we have $\widehat{H}=H$. Hence $d\ast d\widehat{H}=0$ if and only if 
$d\ast dH=0$. 
Hence $\mathfrak{f}$ is a GHIMC surface if and only if $\widehat{\mathfrak{f}}$ is.

If $\mathfrak{f}$ is a HIMC surface, 
then the image of $\widehat{\mathfrak{f}}$ is contained in $\Img\mathbb{H}$ up to translations in $\mathbb{H}$ by Corollary \ref{HIMC}. 
Hence $\widehat{\mathfrak{f}}$ is a HIMC surface. 
Similarly, if $\widehat{\mathfrak{f}}$ is a HIMC surface, then $\mathfrak{f}$ is a HIMC surface.
\end{proof}

Comparing a Willmore conformal map with a GHIMC surface,  
we advance the theory of GHIMC surfaces. 

Define one-forms $A$, $Q$ and $w$ by \eqref{Hopf} and \eqref{Hopfw}. 
The map $\mathfrak{f}$ is Willmore if and only if $w$ is closed (see \cite{BFLPP02}). 
Theorem \ref{cond} is a counterpart in GHIMC surfaces. 
\begin{proof}[Proof of Theorem \ref{cond}]
We have 
\begin{gather*}
-H^{-1}\ast w\, H^{-1}=H^{-1}(-\ast dH+H\,d\mathfrak{f}\,H+R\,dH-H\,dN)H^{-1}\\
=\ast\, dH^{-1}+d\mathfrak{f}+H^{-1}R\,dH\,H^{-1}-dN\,H^{-1}\\
=\ast\, dH^{-1}+d\mathfrak{f}+NH^{-1}\,dH\,H^{-1}-dN\,H^{-1}\\
=\ast\, dH^{-1}+d\mathfrak{f}-N\,dH^{-1}-dN\,H^{-1}\\
=\ast\, dH^{-1}+d(\mathfrak{f}-NH^{-1}).
\end{gather*}
Then $H^{-1}\ast\,w\,H^{-1}+\ast\, dH^{-1}$ is closed. 
Hence $L$ is a GHIMC surface if and only if there exists $n\in\mathbb{H}$ with $n\neq 1$ 
such that $H^{-1}\ast w\,H^{-1}+n\ast \,dH^{-1}$ is closed.  
\end{proof}
We explain Theorem \ref{cond} in a different way. 
Set $v_L:=v_0+v_\infty\mathfrak{f}$. Then, $(v_\infty,v_L)$ is a frame of $\underline{V}$. 
Set $\psi:=-\mathfrak{f}v_0^\ast+v_\infty^\ast$. Then $( \psi,v_0^\ast)$ is the dual frame of $(v_\infty,v_L)$. 
\begin{corollary}\label{HopfC}
Let $L$ be a line bundle of a conformal map with mean curvature sphere $S$ such that
$L_p\neq \{v_\infty\lambda\,|\,\lambda\in\mathbb{H}\}$ for every $p\in M$.  
Set $\mathfrak{f}:=\sigma(L)\colon M\to\mathbb{H}$, $v_L:=v_0+v_\infty\mathfrak{f}$ and 
$\psi:=-\mathfrak{f}v_0^\ast+v_\infty^\ast$. 
Define a map $H\colon M\to\mathbb{H}$ by \eqref{mcs}, a one-form $Q$ by \eqref{Hopf} and a one-form $w$ by \eqref{Hopfw}. 
The map $\mathfrak{f}\colon M\to\mathbb{H}$ is a GHIMC surface if and only if 
$H^{-1}\psi (Qv_\infty)H^{-1}$ is closed. 
\end{corollary}
\begin{proof}
We have 
\begin{gather*}
4H^{-1}\psi (Qv_\infty)H^{-1}=-2\ast dH^{-1}-H^{-1}\ast w\, H^{-1}.
\end{gather*}
Hence the corollary follows by Theorem \ref{cond}. 
\end{proof}

Let $\mathfrak{b}\colon M\to \mathbb{H}$ be a map such that $d\mathfrak{b}=w/2-dH$. 
If $\mathfrak{f}$ is Willmore, then the subspace spanned by 
$v_L$ and $v_L\mathfrak{b}$ defines a B\"{a}cklund transform of $L$ and the map $\mathfrak{b}$ is called 
a backward B\"{a}cklund transform (see \cite{BFLPP02}). 
Theorem \ref{Btrans} is a counterpart of a backward B\"{a}cklund transform of a Willmore surface in GHIMC surfaces. 
\begin{proof}[Proof of Theorem \ref{Btrans}]
Define $Q$ by \eqref{Hopf}. 
Then 
\begin{gather*}
2\,d\mathfrak{h}=-d\mu+d(\mathfrak{f}-N\,H^{-1})\\
=-2\ast dH^{-1}-H^{-1}\ast w\,H^{-1}
=4H^{-1}\psi(Qv_\infty)H^{-1}. 
\end{gather*}
We have 
$\psi (Qv_\infty)=\ast\,(2\,dH-w)$. 
By Lemma \ref{eta}, we have $(\psi (Qv_\infty))^{-N}=0$. 
Hence 
\begin{gather*}
\ast\,(H^{-1}\psi (Qv_\infty)H^{-1})=H^{-1}\ast(\psi (Qv_\infty))H^{-1}\\
=H^{-1}\psi (Qv_\infty)NH^{-1}=H^{-1}\psi (Qv_\infty)H^{-1}R. 
\end{gather*}
Hence $(d\mathfrak{h})^{-R}=0$. 
Then $(d\overline{\mathfrak{h}})_R=0$. 
By Lemma \ref{cBtrans}, the subspace spanned by $v_L$ and $v_L\overline{\mathfrak{h}}$ 
defines a B\"{a}cklund transform of $L$. 
\end{proof}

As a counterpart in a GHIMC surface, we introduce the notion of a backward B\"{a}cklund transform of a GHIMC surface as follows. 
\begin{definition}\label{defbBtrans}
Let $L$ be a line bundle of a conformal map with mean curvature sphere $S$ such that
$L_p\neq \{v_\infty\lambda\,|\,\lambda\in\mathbb{H}\}$ for every $p\in M$.  
We assume that $\mathfrak{f}:=\sigma(L)$ is a GHIMC surface. 
Define $H\colon M\to\mathbb{H}$ and $N\colon M\to S^2$ by \eqref{mcs}. 
Let $\mu\colon M\to\mathbb{H}$ be a map such that $d\mu=\ast\,dH^{-1}$ and 
$\mathfrak{h}\colon M\to\mathbb{H}$ be a map defined by 
\begin{gather*}
\mathfrak{h}=\frac{1}{2}(-\mu+\mathfrak{f}-NH^{-1}).
\end{gather*}
Then we call the map $\overline{\mathfrak{h}}$ the \textit{backward B\"{a}cklund transform} of $\mathfrak{f}$. 
\end{definition}
A backward B\"{a}cklund transform of a Willmore surface is Willmore.  
On the other hand, it is unclear whether 
a backward B\"{a}cklund transform of a GHIMC surface is a GHIMC surface.  

Let $B\colon \mathbb{H}\to\mathbb{H}$ be a Euclidean motion of $\mathbb{E}^4\cong\mathbb{H}$. 
Then there exist $r$, $s\in\mathbb{H}$ with $|r|=|s|=1$ and $t\in\mathbb{H}$ such that 
$Ba=ras^{-1}+t$ for any $a\in\mathbb{H}$. 
\begin{lemma}\label{EMBt}
Assume that $\mathfrak{f}\colon M\to\mathbb{H}$ is a GHIMC surface with mean curvature sphere $S$.  
Define $N\colon M\to S^2$ and $H\colon M\to\mathbb{H}$ by \eqref{mcs}. 
Let $\overline{\mathfrak{h}}=(-\overline{\mu}+\overline{\mathfrak{f}}-\overline{NH^{-1}})/2$ a backward B\"{a}cklund transform of $\mathfrak{f}$ with $d\mu=\ast\,dH^{-1}$. 
Then $\overline{B(2\mathfrak{h})/2}$ is a backward B\"{a}cklund transform of $B\mathfrak{f}$. 
\end{lemma}
\begin{proof}
The mean curvature sphere $\tilde{S}$ of $B\mathfrak{f}$ defines the functions 
$\tilde{N}\colon M\to S^2$, $\tilde{R}\colon M\to S^2$ and $\tilde{H}\colon M\to\mathbb{H}$ by 
\begin{gather}
\begin{gathered}
v_{\tilde{L}}=v_0+v_\infty B\mathfrak{f},\\
\tilde{S}v_\infty=v_\infty \tilde{N}+v_{\tilde{L}}(-\tilde{H}),\enskip \tilde{S}v_{\tilde{L}}=v_{\tilde{L}}(-\tilde{R}),\\
\tilde{N}^2=\tilde{R}^2=-1,\enskip 
\tilde{R}\tilde{H}=\tilde{H}\tilde{N},\\
d(B\mathfrak{f})\,\tilde{H}=(d\tilde{N})_{\tilde{N}},\enskip 
\tilde{H}\,d(B\mathfrak{f})=(d\tilde{R})^{-\tilde{R}}. 
\end{gathered}
\end{gather}
Then $\tilde{N}=rNr^{-1}$ and $\tilde{R}=sRs^{-1}$. 
Let $\tilde{\mathcal{H}}$ be the mean curvature vector of $B\mathfrak{f}$. 
Then $\tilde{\mathcal{H}}=r\mathcal{H}s^{-1}$. 
Hence 
\begin{gather*}
\tilde{H}=-\overline{\tilde{\mathcal{H}}}\tilde{N}=-\overline{r\mathcal{H}s^{-1}}rNr^{-1}=-s\overline{\mathcal{H}}Nr^{-1}=-sHr^{-1}.
\end{gather*}
Set $\tilde{\mu}:=r\mu s^{-1}$. 
Then $d\tilde{\mu}=\ast \,d\tilde{H}^{-1}=r\,d\mu\,s^{-1}$. 
Hence
\begin{gather*}
B(2\mathfrak{h})=r\left(-\mu+\mathfrak{f}-NH^{-1}\right)s^{-1}+t\\
=-r\mu s^{-1}+r\mathfrak{f}s^{-1}-rNr^{-1}rH^{-1}s^{-1}+t\\
=-\tilde{\mu}+B\mathfrak{f}-\tilde{N}\tilde{H}^{-1}. 
\end{gather*}
Then $\overline{B(2\mathfrak{h})/2}$ is a backward B\"{a}cklund transform of $B\mathfrak{f}$. 
\end{proof}

We can construct a Darboux transform of a GHIMC surface from the backward B\"{a}cklund transform. 
\begin{proof}[Proof of Theorem \ref{DTGHIMC}]
Because $v_L$ and $v_L\overline{\mathfrak{h}}$ defines a 
B\"{a}cklund transform of $L$ by Theorem \ref{Btrans}, there exists 
a map $\lambda_\infty\colon M\to\mathbb{H}$ such that $d\lambda_\infty+d\mathfrak{f}\,\overline{\mathfrak{h}}=0$ and 
the map $\widehat{\mathfrak{f}}=\lambda_\infty\overline{\mathfrak{h}}^{-1}+\mathfrak{f}$ is a Darboux transform of $\mathfrak{f}$ with respect to $v_\infty\lambda_\infty$ by Lemma \ref{Dtrans}. 
\end{proof}

We consider a HIMC surface of revolution. 
Let $\mathcal{R}$ be a group consists of all the elements of the rotation group of $\mathbb{H}$ which 
fix the origin and a plane containing the origin, $1$ and $k$. 
A conformal map $\mathfrak{f}\colon M\to\Img\mathbb{H}$ invariant under $\mathcal{R}$ is called 
a surface of revolution. 
It is known that a surface of revolution is an isothermic surface. 
\begin{corollary}\label{Dtransr}
Let $\mathfrak{f}\colon M\to\mathbb{H}$ be an isothermic GHIMC surface invariant under $\mathcal{R}$ 
and $\{\tau_B\}_{B\in\mathcal{R}}$ be the subgroup of holomorphic transformations of $M$ such that 
$B\mathfrak{f}=\mathfrak{f}\circ\tau_B$ for each $B\in\mathcal{R}$. 
Assume that $\lambda_\infty\colon M\to\mathbb{H}$ and $\lambda_L\colon M\to\mathbb{H}$ are maps 
such that $d\lambda_\infty+d\mathfrak{f}\,\lambda_L=0$, 
$\widehat{\mathfrak{f}}:=\lambda_\infty\lambda_L^{-1}+\mathfrak{f}$ is a classical 
Darboux transform and $B\lambda_\infty=\lambda_\infty\circ\tau_B$ for every $B\in\mathcal{R}$. 
Then $\widehat{\mathfrak{f}}$ is a GHIMC surface invariant under $\mathcal{R}$. 
If $\mathfrak{f}\colon M\to\Img\mathbb{H}$ is an isothermic HIMC surface invariant under $\mathcal{R}$, then 
$\widehat{\mathfrak{f}}$ is a HIMC surface invariant under $\mathcal{R}$. 
\end{corollary}
\begin{proof}
By Corollary \ref{inv}, the map $\widehat{\mathfrak{f}}$ is invariant under $\mathcal{R}$. 
By Theorem \ref{DT}, the map $\widehat{\mathfrak{f}}$ is a GHIMC surface. 
By Theorem \ref{DT}, if $\mathfrak{f}$ is a HIMC surface, then $\widehat{\mathfrak{f}}$ is a HIMC surface. 
\end{proof}

All HIMC surfaces of revolution are classified in \cite{BEK97}. 
Let $(x,y)$ be a coordinate of $\mathbb{R}^2$ such that 
$x+yi$ is the standard holomorphic coordinate $\mathbb{C}$. 
\begin{theorem}[\cite{BEK97}, Lemma 2, Theorem 1]\label{cHIMC}
If $\mathfrak{f}\colon\mathbb{C}\to \Img\mathbb{H}$ with the induced metric on $\mathbb{C}$ 
is a surface of revolution, 
then it has a parametrization 
\begin{gather}
\begin{gathered}
\mathfrak{f}(x+yi)=\frac{e^{u(x)/2}\cos (ay)}{a}i+\frac{e^{u(x)/2}\sin (ay)}{a}j+\frac{c(x)}{a}k, \\
u\colon \mathbb{R}\to\mathbb{R},\enskip c\colon \mathbb{R}\to\mathbb{R}, \enskip a\in\mathbb{R}\setminus\{0\}.
\end{gathered}
\label{HIMCR}
\end{gather}
If $\mathfrak{f}\colon  \{x+yi\in\mathbb{C}\,|\,x>0\}\to \Img\mathbb{H}$ with the above parametrization is a HIMC surface, 
then a real-valued function $\phi\colon \mathbb{R}\to\mathbb{R}$ 
defined by 
\begin{gather}
\sin(\phi(x))=\frac{1}{2}c^\prime(x)e^{-u(x)/2},\enskip \cos(\phi(x))=\frac{1}{4}u^\prime(x).\label{phi}
\end{gather}
is a solution to the Painlev\'{e} III equation in trigonometric form  
\begin{gather}
x[\phi^{\prime\prime}(x)-2\sin(2\phi(x))+\phi^\prime(x)+2\sin(\phi(x))]=0. \label{PIII}
\end{gather}

Conversely, let $\phi(x)$ be arbitrary solution to \eqref{PIII} with
$\phi^\prime(x)+2\sin(\phi(x))\neq 0$. 
Then the map $\mathfrak{f}\colon \mathbb{C}\to \Img\mathbb{H}$ defined by \eqref{HIMCR} 
and 
\begin{gather*}
e^{u(x)}=\frac{x^2(\phi^\prime(x)+2\sin(\phi(x)))^2}{4},\\
c(x)=-\frac{x^2}{4}\{(\phi^\prime(x))^2-4[\sin(\phi(x))]^2\}. 
\end{gather*}
is a HIMC surface of revolution. 
\end{theorem}
Then we have a transform of a solution to the Painlev\'{e} III equation in trigonometric form. 
\begin{proof}[Proof of Theorem \ref{tPIII}]
A HIMC surface $\mathfrak{f}$ in Theorem \ref{cHIMC} is invariant under $\mathcal{R}$. 
Because a surface of revolution is an isothermic surface, 
a classical Darboux transform $\widehat{\mathfrak{f}}$ is 
a HIMC surface. 
A HIMC surface $\widehat{\mathfrak{f}}$ is  invariant under $\mathcal{R}$ by Corollary \ref{Dtransr}. 
The function $\widehat{\phi}(x)$ is defined by \eqref{phi} for $\widehat{\mathfrak{f}}$ and 
it is a solution to the the Painlev\'{e} III equation in trigonometric form \eqref{PIII} by Theorem \ref{cHIMC}. 
\end{proof}



\begin{thebibliography}{99}
\bibitem{Bobenko94}
Bobenko, A. I., 
\textit{Surfaces in terms of {$2$} by {$2$} matrices. {O}ld and new
integrable cases}, 
Harmonic maps and integrable systems, 
Aspects Math. E23, 
83--127,
Friedr. Vieweg, Braunschweig, 
1994. 
\bibitem{BE00}
Bobenko, Alexander I. and Eitner, Ulrich, 
\textit{Painlev\'e equations in the differential geometry of surfaces}, 
Lecture Notes in Mathematics 1753,
Springer-Verlag,
Berlin,
2000.
\bibitem{BEK97}
Bobenko, A. and Eitner, U. and Kitaev, A.,
\textit{Surfaces with harmonic inverse mean curvature and {P}ainlev\'e
              equations},
Geom. Dedicata 
68 (1997),
no. 2,
187--227.
\bibitem{Bohle10}
Bohle, Christoph, 
\textit{Constrained {W}illmore tori in the 4-sphere},
J. Differential Geom. 86 (2010), no. 1, 71--131. 
\bibitem{BLPP12}
Bohle, Christoph, Leschke, Katrin, Pedit, Franz and Pinkall, Ulrich, 
\textit{Conformal maps from a 2-torus to the 4-sphere},  
J. Reine Angew. Math. 671 
(2012),
1--30.
\bibitem{BPP09}
Bohle, Christoph, Pedit, Franz and Pinkall, Ulrich,
\textit{The spectral curve of a quaternionic holomorphic line bundle
              over a 2-torus},
Manuscripta Math. 130 (2009),  
no. 3, 
311--352. 
\bibitem{BFLPP02}
Burstall, F. E. and Ferus, D. and Leschke, K. and Pedit, F.
              and Pinkall, U.,
\textit{Conformal geometry of surfaces in {${\it S}^4$} and
              quaternions}, 
Lecture Notes in Mathematics 
1772,
Springer-Verlag,
Berlin,
2002. 
\bibitem{CLP}
Carberry, E., Leschke, K. and Pedit, F.,  
\textit{Darboux transforms and spectral curves of constant mean curvature surfaces revisited},  
to appear in Ann. Glob. Anal. Geom., DOI:10.1007/s10455-012-9347-8. 
\bibitem{HP97}
Hertrich-Jeromin, Udo and Pedit, Franz,
\textit{Remarks on the {D}arboux transform of isothermic surfaces},
Doc. Math. 2 (1997), 
313--333. 
\bibitem{KPP98}
Kamberov, George, Pedit, Franz and Pinkall, Ulrich,
\textit{Bonnet pairs and isothermic surfaces},
Duke Math. J. 92 (1998), no. 3, 637--644. 
\bibitem{Korotkin99}
Korotkin, D. A., 
\textit{On some integrable cases in surface theory}, 
Differential geometry, Lie groups and mechanics. Part~15--1, 
Zap. Nauchn. Sem. POMI 234, 65--124, POMI, 
St.~Petersburg, 1996. 
\bibitem{LP08}
Leschke, Katrin and Pedit, Franz,
\textit{Sequences of {W}illmore surfaces}, 
Math. Z. 259 (2008), no. 1,
113--122. 
\bibitem{LP05}
Leschke, K. and Pedit, F., 
\textit{B\"acklund transforms of conformal maps into the 4-sphere}, 
P{DE}s, submanifolds and affine differential geometry,
Banach Center Publ.,
69,
103--118,
Polish Acad. Sci.,
Warsaw,
2005. 
\bibitem{PP98}
Pedit, Franz and Pinkall, Ulrich,
\textit{Quaternionic analysis on {R}iemann surfaces and differential
              geometry},
Proceedings of the {I}nternational {C}ongress of
              {M}athematicians, {V}ol. {II} ({B}erlin, 1998),
Doc. Math. (1998),
Extra Vol. II,
389--400 (electronic). 
\bibitem{Sym85}
Sym, Antoni,
\textit{Soliton surfaces and their applications (soliton geometry from
              spectral problems)},
Geometric aspects of the {E}instein equations and integrable
              systems ({S}cheveningen, 1984),
Lecture Notes in Phys. 239,
154--231,
Springer,
Berlin,
1985.
\end{thebibliography}
\end{document}